\documentclass[11pt,oneside]{amsart}
\usepackage[utf8]{inputenc}

\usepackage{amsmath, comment}
\usepackage{amsthm, amsfonts, amssymb,latexsym, color}
\title{A convex set with a rich difference}
\author{}
\usepackage{amsmath}
\usepackage[a4paper, total={6in, 8in}]{geometry}
\newtheorem{claim}{Claim}

\numberwithin{exercise}{subsection}

\newtheorem{theorem}{Theorem}
\newtheorem*{theorem*}{Theorem}

\author[O. Roche-Newton]{Oliver Roche-Newton} \address{Institute for Algebra, Johannes Kepler Universit\"{a}t\\
Linz, Austria}
\email{o.rochenewton@gmail.com}

\author[A. Warren]{Audie Warren} \address{Johann Radon Institute for Computational and Applied Mathematics\\
Linz, Austria}
\email{audie.warren@oeaw.ac.at}

\begin{document}
\maketitle

\begin{abstract}

We construct a convex set $A$ with cardinality $2n$ and with the property that an element of the difference set $A-A$ can be represented in $n$ different ways. We also show that this construction is optimal by proving that for any convex set $A$, the maximum possible number of representations an element of $A-A$ can have is $\lfloor |A|/2 \rfloor $.
\end{abstract}

\subsection*{Introduction}
A finite set $A \subset \mathbb R$ is said to be \textit{convex} if the consecutive differences are strictly increasing. That is, if we write $A=\{a_1<a_2<\dots<a_n \}$, $A$ is convex if
\[
a_i - a_{i-1} < a_{i+1} - a_i
\]
holds for all $2 \leq i \leq n-1$. One can also use the equivalent formulation that a set $A$ is convex if we can write $A=f(\{1,2,\dots, n \})$ for some strictly convex function $f$. The convexity of $f$ disrupts the additive structure of the pre-image $\{1,2,\dots, n \}$, and this leads us to expect that a convex set cannot have much additive structure.

This principle can be quantified in different ways, and one such way is to prove that the difference set
\[
A-A:= \{ a-b: a,b \in A\} 
\]
is large. The current state of the art for this problem is a result of Schoen and Shkredov \cite{SS}, proving that the bound\footnote{Throughout this note, the notation
 $X\gg Y$ and $Y \ll X,$ are equivalent and mean that $X\geq cY$ for some absolute constant $c>0$.}
\[
|A-A| \gg |A|^{8/5 - o(1)}
\]
holds for any convex set $A$.

Another approach is to consider the additive energy
\[
E(A):=| \{(a,b,c,d) \in A^4 : a-b=c-d \}|,
\]
which can also be expressed as
\[
E(A)= \sum_{x} r^2_{A-A}
\]
where $r_{A-A}(x):= |\{(a,b) \in A \times A : a-b = x \}|$. The bound
\begin{equation} \label{Ebound}
E(A) \ll |A|^{5/2}
\end{equation}
was proven using incidence theory by Konyagin \cite{K} and using elementary methods by Garaev \cite{G}. See also \cite{RRNS} for an alternative presentation of a proof of \eqref{Ebound}. A further improvement was later given by Shkredov \cite{S}, using additional higher energy tools from additive combinatorics.

One might even expect that a qualitatively stronger statement than \eqref{Ebound} holds; namely that $r_{A-A}(x)$ is guaranteed to be small for all $x \neq 0$. Indeed, if one knew, for instance, that $r_{A-A}(x) \leq |A|^{1-c}$ holds for all $x \neq0$, this immediately implies the non-trivial bound $E(A) \ll |A|^{3-c}$, which in turn implies the non-trivial bound $|A-A| \gg |A|^{1+c}$.

However, a construction of Schoen \cite{Sch} shows that such a uniform upper bound for the representation function $r_{A-A}(x)$ is not possible. Schoen constructed a convex set with $n$ elements and some $x \neq 0$ with $r_{A-A}(x) \geq n/4$.

The main purpose of this note is to give a construction of a convex set with a rich difference which improves the construction of Schoen. We prove the following result.

\begin{theorem} \label{thm:construction}
For every $m\in \mathbb N$, there exists a convex set $A\subseteq \mathbb R$ of size $2m$ and a non-zero element $d \in A-A$ such that $r_{A-A}(d) \geq m$.
\end{theorem}

We also show that this construction is optimal, proving that, for any convex set with cardinality $n$ and any $d \neq 0$,
\[
r_{A-A}(d) \leq \left \lfloor \frac{n}{2} \right \rfloor .
\]

\subsection*{The construction}

\begin{proof}[Proof of Theorem \ref{thm:construction}]

We give a concrete construction of the set 
\[
A=\{a_1<a_2< \dots <a_{2m} \},
\] 
which is made up of two halves. The set $A$ begins with $0$, and then has gaps $1 + (i-1)\delta$, for some very small $\delta >0$ which will be specified later. The first half of $A$ is filled like this. That is, for $1 \leq k \leq m+1$, we define
\[
a_k:= (k-1) + \delta \frac{(k-2)(k-1)}{2},
\]
and so the first $m+1$ elements of $A$ are the elements of the set
\[
A_1:=\left\{ 0, 1, 2+\delta, 3+3\delta,...,m + \delta \frac{m(m-1)}{2}\right\}.
\]
Fix 
\[
d:=m + \delta \frac{m(m-1)}{2}=a_{m+1}.
\]

The rest of $A$ is defined iteratively. For $1 \leq i \leq m-1$, we set
\[
a_{m+1+i}:= a_{1 + 2i}+d.
\]
This immediately gives rise to the system of equations
\begin{equation} \label{eq:diff}
     d=a_{m+1} - a_1=a_{m+2} - a_3= \dots = a_{2m} - a_{2m-1}.
\end{equation}
We therefore have $r_{A-A}(d) \geq m$.

It remains to check that this set is convex. Note that the first part of $A$, namely $A_1=\{a_1,\dots,a_{m+1} \}$, is convex, since the consecutive difference increase by $\delta$ at each step. 

We will prove by induction on $i$ that the set
\[
\{ a_1,a_2,\dots, a_{m+2+i} \}
\]
is convex for $0 \leq i \leq m-2$.

We first check the base case $i=0$. We need to verify that the difference $a_{m+2} - a_{m+1}$ is sufficiently large, which will give a condition on $\delta$. We must have
$$a_{m+2} - a_{m+1} > a_{m+1} - a_m,$$
which upon plugging in the definitions yields
$$2 + \delta > 1 + \frac{\delta m(m-1)}{2} - \frac{\delta (m-2)(m-1)}{2}.$$
After simplification, this gives the condition $\delta < \frac{1}{m-2}$.

Now let $1 \leq i \leq m-2$. We must verify that 
\[
\{ a_1,a_2,\dots, a_{m+2+i} \}
\]
is convex, given the induction hypothesis that $\{ a_1,a_2,\dots, a_{m+1+i} \}$ is convex. All that remains is to check that
$$a_{m+2+i} - a_{m+1+i} > a_{m+1+i} - a_{m+i}.$$
 We use equations \eqref{eq:diff} to rewrite each side, as
$$a_{m+2+i} - a_{m+1+i} = a_{1+2(i+1)} - a_{1 + 2i}$$
$$a_{m+1+i} - a_{m+i} = a_{1+2i} - a_{1+2(i-1)}.$$
Note that, since the differences on the right hand side above are then consecutive differences of length two within a convex set, we have
\begin{align*}
   a_{m+2+i} - a_{m+1+i} &= a_{1+2(i+1)} - a_{1 + 2i} \\
   & > a_{1+2i} - a_{1+2(i-1)} \\
   & = a_{m+1+i} - a_{m+i} 
\end{align*}
as needed. Here we have used the inductive hypothesis that $\{ a_1,a_2,\dots, a_{m+1+i} \}$ is convex as well as the fact that $1+2(i+1) \leq m+1+i$. The latter inequality follows from the condition that $i \leq m-2$.
\end{proof}
Note that by taking $\delta$ to be a sufficiently small rational number, and dilating the set $A$ through by common denominators, we can find $A \subseteq \mathbb Z$ satisfying Theorem \ref{thm:construction}.

\subsection*{A matching upper bound for the representation function}

The next result shows that the construction of Theorem \ref{thm:construction} is optimal.

\begin{theorem}
For a convex set $A \subset \mathbb R$ and any $d \in \mathbb R \setminus \{0\}$,
\[
r_{A-A}(d) \leq \left \lfloor \frac{|A|}{2} \right \rfloor .
\]
\end{theorem}

\begin{proof}

Write the elements of $A$ in increasing order so that $A=\{a_1<a_2< \dots < a_n \}$. Suppose that $d$ can be represented in $t$ different ways as an element of $A-A$. We can write
\begin{align}\label{list}
    d = & a_{j_1+k_1} - a_{j_1} \nonumber
    \\= & a_{j_2+k_2} - a_{j_2} \nonumber
    \\ & \vdots \nonumber
    \\= & a_{j_t+k_t} - a_{j_t}
\end{align}
such that the $k$ indices satisfy
\begin{equation} \label{kdefn}
k_1 > k_2 > \dots > k_t.
\end{equation}
Indeed, because $A$ is convex, we cannot have two of the $k$ indices repeating in the list \eqref{list}. This follows from the fact that, for fixed $k$, the sequence 
\begin{equation} \label{inc1}
(a_{j+k} - a_{j})_{j \in \mathbb N}
\end{equation}
is strictly increasing. Note also that, for fixed $j$, the sequence
\begin{equation} \label{inc2}
(a_{j+k} - a_{j})_{k \in \mathbb N}
\end{equation}
is strictly increasing. This follows immediately from the fact that the $a_i$ are increasing.

\begin{claim} For all $1 \leq i \leq t-1$
\[
j_{i+1} \geq j_{i} +2
\]
\end{claim}

\begin{proof}[Proof of Claim]
Suppose for a contradiction that $j_{i+1} \leq j_{i} +1$. We also have $k_{i+1} \leq k_{i}-1$, and so $j_{i+1}+k_{i+1} \leq j_{i}+k_{i}$. Therefore
\[
a_{j_{i+1}+k_{i+1}} \leq a_{j_{i}+k_{i}}.
\]
But then it follows from \eqref{list} that
\[
0 \leq a_{j_i+k_i} -  a_{j_{i+1}+k_{i+1}} = a_{j_i} - a_{j_{i+1}},
\]
and so
\begin{equation} \label{jcompare}
j_i \geq j_{i+1}.
\end{equation}
However, since the sequences \eqref{inc1} and \eqref{inc2} are strictly increasing, it follows that
\begin{align*}
    a_{j_{i+1}+k_{i+1}} - a_{j_{i+1}} & \leq a_{j_{i}+k_{i+1}} - a_{j_{i}} 
    \\& < a_{j_{i}+k_{i}} - a_{j_{i}} .
\end{align*}
This contradicts \eqref{list}.

\end{proof}

Applying the claim iteratively yields
\begin{equation} \label{jiter}
j_t \geq j_{t-1}+2 \geq j_{t-2}+4 \geq \dots \geq j_1+2(t-1) \geq 1+2(t-1)=2t -1.
\end{equation}
We also know that $j_t+k_t \leq n$ and $k_t \geq 1$. Therefore,
\[
j_t \leq n -1.
\]
Combining this with \eqref{jiter} gives
\[
t \leq n/2.
\]
Finally, since $t$ is an integer, this is equivalent to the bound
\[
t \leq \lfloor n/2 \rfloor.
\]

\end{proof}

\subsection*{Concluding remarks}

Interestingly, the construction cannot be modified to give a rich sum in a convex set. For $x \in \mathbb R$, we use the notation
\[
r_{A+A}(x):=|\{ (a,b) \in A \times A : a+b=x \}|.
\]
In sharp contrast with Theorem \ref{thm:construction}, the bound
\begin{equation} \label{reps}
r_{A+A}(C) \ll |A|^{2/3}.
\end{equation}
holds for any convex set $A$ and $C \in \mathbb R$. The inequality \eqref{reps} was also observed by Schoen \cite{Sch}, and can be proved using the Szemer\'{e}di-Trotter Theorem.

Another interesting direction is to determine how many $k$-rich representations can occur. A well-known application of the Szemer\'{e}di-Trotter Theorem (see for instance \cite{RRNS}) gives the bound
\begin{equation} \label{upper}
|\{d : r_{A-A}(d) \geq t \}| \ll \frac{n^3}{t^3}
\end{equation}
for any convex set $A$ with cardinality $n$. On the other hand, one can glue together $n/t$ copies of the construction in Theorem \ref{thm:construction} with $t$ elements in order to obtain a convex set $A$ with $n$ elements and
\begin{equation} \label{lower}
|\{d : r_{A-A}(d) \geq t \}| \gg \frac{n}{t}.
\end{equation}
There is a considerable gap between the upper and lower bounds of \eqref{upper} and \eqref{lower} respectively, although the bounds converge as $t$ gets close to $n$.

\subsection*{Acknowledgements} The authors were supported by the Austrian Science Fund FWF Project P 34180. We are grateful to Brandon Hanson, Misha Rudnev and Dmitrii Zhelezov for helpfully sharing their insights. We are particularly grateful to Ilya Shkredov for informing us about the reference \cite{Sch}.

\end{document}